\documentclass[a4paper,10pt]{article}
\usepackage{graphicx}

\usepackage{geometry}

\usepackage[ngerman]{babel}
\usepackage{amssymb,amsmath}
\usepackage[usenames,dvipsnames]{color}
\usepackage{amsfonts}
\usepackage{tikz}
\usepackage{hyperref}
\hypersetup{colorlinks,citecolor=black,filecolor=black,linkcolor=black,urlcolor=black}

\usepackage{framed}

\usepackage{latexsym}
\usepackage{amsthm}
\usepackage{bm}

\usepackage{amscd}
\usepackage[all]{xy}

\renewcommand{\abstract}{\textbf{Abstract}}
\usepackage{cite}
\newcommand{\I}{\mathrm{i}}
\renewcommand{\d}{{\rm d}}

\DeclareMathOperator{\id}{id}

\newtheorem{satz}{Satz}[subsection]
\newtheorem{theorem}[satz]{\textit{Theorem}}
\newtheorem{lemma}[satz]{\textit{Lemma}}

\newtheorem{proposition}[satz]{\textit{Proposition}}
\newtheorem{korollar}[satz]{\textit{Corollary}}

\numberwithin{equation}{section}
\date{}

\title{\textbf{\textit{Deformations of $\Xi(s)=\Xi(1-s)$ and the heat equation}}}
\author{\textit{Johannes L\"offler}}
\begin{document}
\maketitle\vspace{-0.15cm}
\begin{abstract} This paper studies deformations of the well-known $\Xi(s)=\Xi(1-s)$ equation for the Riemann $\Xi$ function satisfied by $\Xi_\rho(s):=\int_{0}^\infty\frac{{\d}t}{t}{t^{\frac{s}{2}}}\Psi(t)e^{-\rho\ln^2(t)}$ where ${\Psi(t)=\sum_{n\in\mathbb{N}^+}e^{-\pi{n}^2t}}$. \end{abstract}\vspace{-0.2cm}
\subsubsection*{\textit{Introduction}}\vspace{-0.05cm}
Inspired by results of Jensen work of Jensen \cite{PO}, P\'olya \cite{PO1}, Brujin \cite{Bru} and Newman \cite{New} showed that certain classes of trigonometric integrals admit only real roots and proved that the integral kernel multiplication factors $\exp({-\rho\ln^2(t)})$
with $\rho\in\mathbb{R}^+$ conserve and improve the ``{reality}" {of} the roots in the following sense: Suppose $b>2$ and $\varphi(t)$ decays faster then $t^{-b}$ at $\infty$ and $t^{b}$ at $0$. If we have $\overline{\varphi(t)}=\varphi(1/t)$ and the roots $\mathcal{M}\big[\varphi\big](\I{z})=0\Rightarrow\Im(z)<\Delta$ then the roots $\mathcal{M}\big[e^{-\rho\ln^2(\cdot)}\varphi\big](\I{z})=0$ satisfy $\Im(z)<\sqrt{\Delta^2-4\rho^2}$. We refer to the recent articles \cite{Su1},\cite{Su2},\cite{Su3} for further considerations.

The universal factors $e^{-\rho\ln^2(t)}$ also play a key role in our calculations, they serve as a neat regularization, but the spirit of our considerations is different, we concentrate on the inhomogeneity $(t^{-\frac{1}{2}}-1)/2$ in the well known functional equation $\Psi(t)=t^{-\frac{1}{2}}\Psi\left(1/t\right)+(t^{-\frac{1}{2}}-1)/2$ obtained by Poisson summation: Let $\mathcal{M}\left[{f}\right](s)=\int_{0}^\infty\frac{{\d}t}{t}t^s{f(t)}$ denote the Mellin transform. We have the well known Gauss identity
\begin{equation}\label{Integration1}\mathcal{M}\left[e^{-\rho\ln^2(t)}\right](s)=\sqrt{{\pi}/{\rho}}{e}^{\frac{s^2}{4\rho}}\end{equation}
here we substituted $x=\ln(t)$ to rewrite $\mathcal{M}\big[e^{-\ln^2(t)}\big](0)=\int_{-\infty}^\infty\d{x}\;{e}^{-x^2}=\sqrt{\pi}$. Notice in \ref{Integration1} the prototype of an essential singularity in $\rho$ appears.

Consider for $m\in\mathbb{N}$ the sum $\Xi^m_{\rho}(s):=\sum_{0\leq{l}\leq{m}}\Xi_{\rho}(s+l)$ with
$\Xi_{\rho}(s):=\mathcal{M}\big[\Psi\cdot{e}^{-\rho\ln^2(t)}\big](s/2)$.
It is proved in \cite{JL} that $\Xi^m_{\rho}$ satisfy $\forall\rho\in\mathbb{R}^+$ the telescope identity
\begin{align}\label{Para1}&\Xi^m_{\rho}(s)-\Xi^m_{\rho}(1-m-s)=\sqrt{{\pi}/{\rho}}\Big(e^{\frac{(s-1)^2}{16\rho}}-e^{\frac{(s+m)^2}{16\rho}}\Big)\hspace{-0.05cm}\big/2\end{align}
Formula \ref{Para1} implies the equivalences
$\Xi^m_{\rho}(s)-\Xi^m_{\rho}(1-m-s)=0\Leftrightarrow{s}\in\frac{1-m}{2}+16\rho\pi\I\frac{\mathbb{Z}}{1+m}$ and in the same manner we have
$\tilde{\Xi}^m_{\rho}(s)+(-1)^m\tilde{\Xi}^m_{\rho}(1-m-s)=0\Leftrightarrow{s}\in\frac{1-m}{2}+16\rho\pi\I\big(-\frac{1}{2}+\frac{\mathbb{Z}}{1+m}\big)$ where $\tilde{\Xi}^m_{\rho}(s):=\sum_{l=0}^m(-1)^l\tilde{\Xi}_{\rho}(s+l)$ with $\tilde{\Xi}_\rho(s)=(1-2\id)\Xi_\rho(s)+16\rho\partial_s\Xi_\rho(s)$.

One of our main new results, based on the heat equation $\left(\partial_\rho+4\partial^2_s\right)\Xi_{\rho}(s)=0$ and Lagrange's variation of parameters method, is the formula
\begin{align}\label{2DiffSolIntro} e^{\frac{-s^2+{s}}{16\rho}}\Xi_{\rho}(s)=&\frac{e^{\frac{1}{32\rho}}}{4}\sqrt{\frac{\pi}{\rho}}\sinh\left(\frac{\frac{1}{2}-s}{16\rho}\right)+e^{\frac{1}{64\rho}}\Xi_\rho\left(\frac{1}{2}\right)\cosh\left(\frac{\frac{1}{2}-s}{16\rho}\right)\nonumber\\&+\frac{1}{2\rho}\int_{1/2}^{s}\hspace{-0.2cm}\d{t}\;\sinh\left(\frac{s-t}{16\rho}\right)e^{\frac{-t^2+{t}}{16\rho}}\mathcal{M}\big[(2t^2\partial_t^2\Psi+3t\partial_t\Psi){e}^{-\rho\ln^2(t)}\big]\left(\frac{t}{2}\right)\end{align}
and the remaining integral \ref{2DiffSolIntro} is symmetric with respect to $s\rightarrow1-s$. We also sketch how to iterate this procedure and have an analogous result \ref{H4DiffSol2} for $\tilde{\Xi}_{\rho}$ and for the function $\mathcal{M}\big[(2t^2\partial_t^2\Psi+3t\partial_t\Psi){e}^{-\rho\ln^2(t)}\big]\left(\frac{s}{2}\right)$ that motivated the considerations of P\'olya, Brujin and Newman \ref{PBN}. 

The functional equation \ref{Para1} for $m=0$ also hides behind more involved expressions obtained by a study of the family
\begin{align}\label{EquiSym4}
\Xi(\rho,\vec{s},\vec{\lambda}\;):=\int_{0}^\infty\frac{{\d}t_1}{t_1}\cdots\int_{0}^\infty\frac{{\d}t_d}{t_d}\left(\prod_{i=1}^d{t^{{s_i}/2}_i}\Psi^{\lambda_i}(t_i)\right)e^{-\sum_{1\leq{i,j}\leq{d}}\rho_{ij}\ln(t_i)\ln(t_j)}
\end{align}
where $\rho$ is a symmetric $d\times{d}$ matrix underlying certain convergence restrictions, $\vec{s}\in\mathbb{C}^d$ and $\vec{\lambda}\in\mathbb{C}^d$. The case $\vec{\lambda}=0$ is standard \ref{LowEnd} and in this paper we only study $\Xi(\rho,\vec{s}):=\Xi(\rho,\vec{s},(1,\cdots,1))$, but algebraic the other cases are also suitable for analogous computations if $\vec{\lambda}\in\mathbb{N}^d$. To consider the general case $\vec{\lambda}\in\mathbb{C}^d$ is motivated by the van der Geer and Schoof  two-variable zeta function \cite{LaJ}. 

Writing $t^{s}=e^{s\ln(t)}$ the vector $\vec{s}\in\mathbb{C}^n$ can be interpreted as a linear term while in this sense $\rho$ corresponds to a quadratic term in the exponent of the integral kernel. Notice we have the symmetric interpretations $t_i^{-\rho_{ij}\ln(t_j)}=e^{-\rho_{ij}\ln(t_i)\ln(t_j)}=t_j^{-\rho_{ij}\ln(t_i)}$.

Some evidence for a connection of the Riemann zeros and eigenvalues of large random matrices has been put forward, we refer the reader for example to the reference \cite{Ka}. Functional equations for the family $\Xi(\rho,\vec{s}\;)$ obtained by integration, maybe in combination with the heat equations
$\Big(\partial_{\rho_{ij}}+\frac{8}{1+\delta_i^j}\partial^2_{s_is_j}\Big)\Xi(\rho,\vec{s})=0$
could imply some insights. 
\section{\textit{Comments on the family $\Xi(\rho,\vec{s},\vec{n}\;)$}}\label{FirstSection}
\subsection{\textit{The $\exp$ part of the family}}\label{LowEnd}
The multi-dimensional generalization of \ref{Integration1} is the well-known Gaussian integral identity 
\begin{align*}
\mathrm{e}(\rho,\vec{s}\;):=&\int_{0}^\infty\frac{{\d}t_1}{t_1}\cdots\int_{0}^\infty\frac{{\d}t_d}{t_d}\left(\prod_{i=1}^d{t^{{s_i}/2}_i}
\right)e^{-\sum_{1\leq{i,j}\leq{d}}\rho_{ij}\ln(t_i)\ln(t_j)}
=\sqrt{\frac{\pi^d}{\det(\rho)}}\exp\left(\frac{\vec{s}\cdot\rho^{-1}\vec{s}}{16}\right)
\end{align*}
Notice $\mathrm{e}(\rho,\vec{s}\;)=\Xi(\rho,\vec{s},\vec{0})$, we just use this notation to avoid redundant symbols. Clearly if for some $\rho$ this integrals converge absolute $\forall\vec{s}\in\mathbb{C}^d$ then also \ref{EquiSym4} will converge absolute. If we choose $\rho_{ij}$ with $i\neq{j}$ purely imaginary we can estimate $\vert\Xi(\rho,\vec{s}\;)\vert\leq\prod_{i=1}^d\Xi_{\Re(\rho_{ii})}(\Re(s_i))$ and hence the integral clearly converges if all diagonal entries satisfy $\Re(\rho_{ii})>0$.

$(\rho,\vec{s})\sim(\tilde{\rho},\vec{\tilde{s}})\Leftrightarrow\exists{\lambda_1,\cdots,\lambda_d\in\mathbb{R}^+}:({s}_i=\tilde{s}_i/\lambda_i)\wedge({\rho}_{kl}=\tilde{\rho}_{kl}/\lambda_k\lambda_l)$ defines a equivalence relation and by substitution this integrals essentially only depend on their equivalence class:
$$\mathrm{e}(\rho,\vec{s}\;)=\frac{1}{\prod_{i=1}^{d}\lambda_i}\mathrm{e}\left(\left(\begin{matrix}\rho_{11}/\lambda_1^2&.&\rho_{1d}/\lambda_1\lambda_d\\.&.&.&\\\rho_{1d}/\lambda_1\lambda_d&.&\rho_{dd}/\lambda_d^2\end{matrix}\right),\left(\begin{matrix}s_1/\lambda_1\\.\\s_d/\lambda_d\end{matrix}\right)\right)$$
\subsection{\textit{Factorisation conditions}}\label{Factor}
Let us for a symmetric $d\times{d}$ matrix $\rho$ denote some $2\times2$ sub-matrix determinants by
$$\mathcal{R}_{ik}^\rho:=\rho_{ii}\rho_{kk}-\rho_{ik}^2$$
$$\mathcal{T}_{ijk}^{\rho}:=\rho_{ij}\rho_{kk}-\rho_{ik}\rho_{jk}$$
We substitute $\Psi(t_k)=t_k^{-1/2}{\Psi\left(1/t_k\right)}+(t_k^{-1/2}-1)/2$, multiply out the kernel of the integral, use Fubini's theorem and calculate two 1 dimensional integrals with \ref{Integration1} and yield: 
\begin{lemma}\label{BasicRel}
Suppose $\Xi(\rho,\vec{s}\;)$ converges absolutely $\forall\vec{s}\in\mathbb{C}^d$. The $s_k\rightarrow1-s_k$ relation\begin{align}\label{SK->1-SK} 
&\Xi(\rho,\vec{s}\;)=\Xi\left(\left(\begin{matrix}\rho_{11}&.&-\rho_{1k}&.&\rho_{1d}\\.&.&.&.&.\\-\rho_{1k}&.&\rho_{kk}&.&-\rho_{kd}\\.&.&.&.&.\\\rho_{1d}&.&-\rho_{kd}&.&\rho_{dd}\end{matrix}\right),\left(\begin{matrix}s_1\\.\\1-s_k\\.\\s_d\end{matrix}\right)\right)\\&+\sqrt{\frac{\pi}{\rho_{kk}}}\frac{e^{\frac{(s_k-1)^2}{16\rho_{kk}}}}{2}\Xi\left(\rho_k,\left(\begin{matrix}{s}_1-(s_k-1)\frac{\rho_{1k}}{\rho_{kk}}\\.\\{s}_{k-1}-(s_k-1)\frac{\rho_{k-1k}}{\rho_{kk}}\\{s}_{k+1}-(s_k-1)\frac{\rho_{k+1k}}{\rho_{kk}}\\.\\{s}_{d}-(s_k-1)\frac{\rho_{dk}}{\rho_{kk}}\end{matrix}\right)\right)-\sqrt{\frac{\pi}{\rho_{kk}}}\frac{e^{\frac{s_k^2}{16\rho_{kk}}}}{2}\Xi\left(\rho_k,\left(\begin{matrix}{s}_1-s_k\frac{\rho_{1k}}{\rho_{kk}}\\.\\{s}_{k-1}-s_k\frac{\rho_{k-1k}}{\rho_{kk}}\\{s}_{k+1}-s_k\frac{\rho_{k+1k}}{\rho_{kk}}\\.\\{s}_{d}-s_k\frac{\rho_{dk}}{\rho_{kk}}\end{matrix}\right)\right)\nonumber
\end{align}
holds, where the $d-1\times{d}-1$ matrix $\rho_k$ is obtained from $\rho$ by deleting the $k$th column and $k$th row and acting on the other coefficients by $\rho_{ii}\rightarrow\mathcal{R}_{ik}^{\rho}/\rho_{kk}$ for the diagonal and $\rho_{ij}\rightarrow\mathcal{T}_{ijk}^{\rho}/\rho_{kk}$ for the off-diagonal entries.
\end{lemma}
Observe if $\rho_{ij}=0$ for $i\neq{j}$ we can isolate the $s_j$ and $\rho_{jj}$ dependence with the product factor $\Xi_{\rho_{jj}}(s_j)$. The factorisation condition $\mathcal{T}_{ijk}^{\rho}=0$ is obviously invariant under complex scalings of $\rho$.
\section{\textit{Some calculations for $d=2$ and $d=3$}}\label{n=2}
\begin{lemma}\label{2DFormulas}  Let $\Re(\rho_{ii})>0$ and $\det(\Re\left(\rho\right))>0$ for $i\in\{1,2\}$.  $\forall\vec{s}\in\mathbb{C}^2$ we have
\begin{align}\label{Fun1}&\Xi\left(\left(\begin{matrix}\rho_{11}&\rho_{12}\\\rho_{12}&\rho_{22}\end{matrix}\right),\left(\begin{matrix}s_1\\s_2\end{matrix}\right)\right)-\Xi\left(\left(\begin{matrix}\rho_{11}&\rho_{12}\\\rho_{12}&\rho_{22}\end{matrix}\right),\left(\begin{matrix}1-s_1\\1-s_2\end{matrix}\right)\right)\\&=\frac{\pi{e}^{\frac{\rho_{22}s_1^2+\rho_{11}s_2^2-2\rho_{12}s_1s_2}{16\det(\rho)}}}{2^2\sqrt{\det(\rho)}}\Big[1+e^{\frac{\rho_{11}+\rho_{22}-2\rho_{12}-2\rho_{22}s_1-2\rho_{11}s_2+2\rho_{12}(s_1+s_2)}{16\det(\rho)}}\nonumber\\&\hspace{3.8cm}-e^{\frac{\rho_{22}-2\rho_{22}s_1+2\rho_{12}s_2}{16\det(\rho)}}-e^{\frac{\rho_{11}-2\rho_{11}s_2+2\rho_{12}s_1}{16\det(\rho)}}\Big]\nonumber\\&+\frac{\sqrt{\pi}}{2\sqrt{\rho_{22}}}\Big[e^{\frac{(s_2-1)^2}{16\rho_{22}}}\Xi_{\det(\rho)/\rho_{22}}\Big(1-s_1-{\frac{\rho_{12}}{\rho_{22}}(1-{s}_2)}\Big)-e^{\frac{s_2^2}{16\rho_{22}}}\Xi_{\det(\rho)/\rho_{22}}\Big(1-s_1+{\frac{\rho_{12}}{\rho_{22}}{s_2}}\Big)\Big]\nonumber\\&+\frac{\sqrt{\pi}}{2\sqrt{\rho_{11}}}\Big[e^{\frac{(s_1-1)^2}{16\rho_{11}}}\Xi_{\det(\rho)/\rho_{11}}\Big(1-s_2-{\frac{\rho_{12}}{\rho_{11}}(1-{s}_1)}\Big)-e^{\frac{s_1^2}{16\rho_{11}}}\Xi_{\det(\rho)/\rho_{11}}\Big(1-s_2+{\frac{\rho_{12}}{\rho_{11}}{s_1}}\Big)\Big]\nonumber
\end{align}
\begin{align}\label{Fun11}&=\frac{\sqrt{\pi}}{2\sqrt{\rho_{11}}}\Big[e^{\frac{(s_1-1)^2}{16\rho_{11}}}\Xi_{\det(\rho)/\rho_{11}}\Big(s_2+{\frac{\rho_{12}}{\rho_{11}}(1-{s}_1)}\Big)-e^{\frac{s_1^2}{16\rho_{11}}}\Xi_{\det(\rho)/\rho_{11}}\Big(s_2-{\frac{\rho_{12}}{\rho_{11}}{s_1}}\Big)\Big]\nonumber\\&+\frac{\sqrt{\pi}}{2\sqrt{\rho_{22}}}\Big[e^{\frac{(s_2-1)^2}{16\rho_{22}}}\Xi_{\det(\rho)/\rho_{22}}\Big(1-s_1-{\frac{\rho_{12}}{\rho_{22}}(1-{s}_2)}\Big)-e^{\frac{s_2^2}{16\rho_{22}}}\Xi_{\det(\rho)/\rho_{22}}\Big(1-s_1+{\frac{\rho_{12}}{\rho_{22}}{s_2}}\Big)\Big]
\end{align}
\end{lemma}
\begin{proof}[Proof]  The proof is analogous to the proof of the more general statement \ref{Result3D}. We have
$\det(\rho)=\rho_{11}\rho_{22}-\rho_{12}^2=\mathcal{R}_{12}^\rho$
and the formula
\begin{equation}\label{2DHelpLemma}\mathrm{e}(\rho,\vec{s}\;)=\pi\exp\left(\frac{\rho_{11}s_2^2+\rho_{22}s_1^2-2\rho_{12}s_1s_2}{16\det(\rho)}\right)\big/\sqrt{\det(\rho)}\end{equation}
In practice the conditions are satisfied if
$\Re(\rho_{11})>0<\Re(\rho_{22})$ and $\Re(\rho_{12})\leq\sqrt{\Re(\rho_{11})\Re(\rho_{22})}$.\end{proof}
\begin{korollar} Let $\Re(\rho_{11})>0$ and $\det(\Re\left(\rho\right))>0$. $\forall{s}\in\mathbb{C}$ we have
\begin{align}\label{FunCor1}&\Xi\left(\left(\begin{matrix}\rho_{11}&\rho_{12}\\\rho_{12}&\rho_{11}\end{matrix}\right),\left(\begin{matrix}s\\s\end{matrix}\right)\right)-\Xi\left(\left(\begin{matrix}\rho_{11}&\rho_{12}\\\rho_{12}&\rho_{11}\end{matrix}\right),\left(\begin{matrix}1-s\\1-s\end{matrix}\right)\right)\\&=\frac{\sqrt{\pi}}{\sqrt{\rho_{11}}}\Big[e^{\frac{(s-1)^2}{16\rho_{11}}}\Xi_{\det(\rho)/\rho_{11}}\Big(1-s-{\frac{\rho_{12}}{\rho_{11}}(1-s)}\Big)-e^{\frac{s^2}{16\rho_{11}}}\Xi_{\det(\rho)/\rho_{11}}\Big(1-s+{\frac{\rho_{12}}{\rho_{11}}{s}}\Big)\Big]\nonumber
\end{align}
 if and only if
$s\in\frac{1}{2}-\frac{\rho_{12}}{2(\rho_{11}-\rho_{12})}-8\frac{\det(\rho)}{\rho_{11}-\rho_{12}}\left[2\pi\I\mathbb{Z}+\ln\left(1\pm\sqrt{1-e^{\frac{-\rho_{12}}{8\det(\rho)}}}\right)\right]$.
 We also have
\begin{align}\label{FunCor2}0=&\quad{e}^{\frac{s^2}{16\rho_{11}}}\Xi_{\det(\rho)/\rho_{11}}\Big(1-s-{\frac{\rho_{12}}{\rho_{11}}s}\Big)-e^{\frac{(1-s)^2}{16\rho_{11}}}\Xi_{\det(\rho)/\rho_{11}}\Big(1-s+{\frac{\rho_{12}}{\rho_{11}}(1-s)}\Big)\\&+e^{\frac{(s-1)^2}{16\rho_{11}}}\Xi_{\det(\rho)/\rho_{11}}\Big(s-{\frac{\rho_{12}}{\rho_{11}}(1-s)}\Big)-e^{\frac{s^2}{16\rho_{11}}}\Xi_{\det(\rho)/\rho_{11}}\Big(s+{\frac{\rho_{12}}{\rho_{11}}{s}}\Big)\nonumber
\end{align}
 if and only if 
$s\in\frac{1}{2}+\frac{\rho_{12}}{2(\rho_{11}+\rho_{12})}-8\frac{\det(\rho)}{\rho_{11}+\rho_{12}}\left[2\pi\I\mathbb{Z}+\ln\left(1\pm\sqrt{1-e^{\frac{\rho_{12}}{8\det(\rho)}}}\right)\right]$.
\end{korollar}
\begin{proof}[Proof] We choose $\rho_{11}=\rho_{22}$. For the first equivalence we consider $s_1=s_2=s$: In this case 2 of the 4 $\exp$ function terms in \ref{Fun1} are identical and the statement equivalent to the quadratic equation
$0=1-2e^{\frac{\rho_{11}}{16\det(\rho)}}e^{\frac{-2\rho_{11}s+2\rho_{12}s}{16\det(\rho)}}+e^{\frac{2\rho_{11}-2\rho_{12}}{16\det(\rho)}}\left(e^{\frac{-2\rho_{11}s+2\rho_{12}s}{16\det(\rho)}}\right)^2$. For the second equivalence we choose  $s_1=s$ and $s_2=1-s$, here we obviously have
$\Xi\left(\rho,\left(s,1-s\right)\right)-\Xi\left(\rho,\left(1-s,s\right)\right)=0$. \end{proof}
The following Lemma also concerns the roots of the inhomogeneity in \ref{Fun1}:
\begin{lemma} If $\rho_{12}=\frac{1}{32\pi\I{n}}\pm\sqrt{\gamma^2-\frac{1}{(32\pi{n})^2}}$ with $n\in\mathbb{Z}$ we have
$$0=1+e^{\frac{2\gamma-2\rho_{12}-2\gamma{s}_1-2\gamma{s}_2+2\rho_{12}(s_1+s_2)}{16(\gamma^2-\rho_{12}^2)}}-e^{\frac{\gamma-2\gamma{s}_1+2\rho_{12}s_2}{16(\gamma^2-\rho_{12}^2)}}-e^{\frac{\gamma-2\gamma{s}_2+2\rho_{12}s_1}{16(\gamma^2-\rho_{12}^2)}}$$
if and only if $s_1=\frac{\gamma}{\frac{1}{32\pi\I{n}}\pm\sqrt{\gamma^2-\frac{1}{(32\pi{n})^2}}}s_2+2\frac{n'}{n}$  or $s_2=\frac{\gamma}{\frac{1}{32\pi\I{n}}\pm\sqrt{\gamma^2-\frac{1}{(32\pi{n})^2}}}s_1+2\frac{n'}{n}$  with $n'\in\mathbb{Z}$.
\end{lemma}
\begin{proof}[Proof] 
We study the symmetric case $\rho_{11}=\rho_{22}=\gamma$ and consider the roots of
$$(X_1X_2)^{-\frac{1}{2}}+e^{\frac{-\rho_{12}}{8(\gamma^2-\rho_{12}^2)}}Z^2(X_1X_2)^{\frac{1}{2}}-2Z\cosh\left(\left(\frac{1}{2}-\frac{\rho_{12}}{-\gamma+\rho_{12}}\right)\ln\left(\frac{X_1}{X_2}\right)\right)$$
where we set $X_i=e^{\frac{-\gamma+\rho_{12}}{8(\gamma^2-\rho_{12}^2)}s_i}$ for $i\in\{1,2\}$ and $Z=e^{\frac{\gamma}{16(\gamma^2-\rho_{12}^2)}}$. Hence we yield
$$Z=\frac{\cosh\left(\left(\frac{1}{2}-\frac{\rho_{12}}{-\gamma+\rho_{12}}\right)\ln\left(\frac{X_1}{X_2}\right)\right)\pm\sqrt{\cosh^2\left(\left(\frac{1}{2}-\frac{\rho_{12}}{-\gamma+\rho_{12}}\right)\ln\left(\frac{X_1}{X_2}\right)\right)-e^{\frac{-\rho_{12}}{8(\gamma^2-\rho_{12}^2)}}}}{e^{\frac{-\rho_{12}}{8(\gamma^2-\rho_{12}^2)}}(X_1X_2)^{\frac{1}{2}}}$$
If $e^{\frac{-\rho_{12}}{8(\gamma^2-\rho_{12}^2)}}=1\Leftrightarrow\rho_{12}=\frac{1}{32\pi\I{n}}\pm\sqrt{\gamma^2-\frac{1}{(32\pi{n})^2}}$ we can rewrite the previous equation with $\cosh^2(x)-1=\sinh^2(x)$
to $e^{\frac{\gamma}{16(\gamma^2-\rho_{12}^2)}}(X_1X_2)^{\frac{1}{2}}=\left(\frac{X_1}{X_2}\right)^{\pm\left(\frac{1}{2}-\frac{\rho_{12}}{-\gamma+\rho_{12}}\right)}$
hence we yield
$2\pi\I{n'}+\frac{-\gamma+\rho_{12}}{16(\gamma^2-\rho_{12}^2)}(s_1+s_2)=\pm(s_1-s_2)\frac{-\gamma-\rho_{12}}{16(\gamma^2-\rho_{12}^2)}$
and this is equivalent to
$s_1=\frac{\gamma}{\rho_{12}}s_2+32\pi\I{n'}\frac{\gamma^2-\rho_{12}^2}{\rho_{12}}\quad\Leftrightarrow\quad{s}_1=\frac{\gamma}{\frac{1}{32\pi\I{n}}\pm\sqrt{\gamma^2-\frac{1}{(32\pi{n})^2}}}s_2+2\frac{n'}{n}$
or the equation where we exchange $1$ and $2$.\end{proof} 
\begin{proposition}\label{InEq}
Let $\Re(\rho_{ii})>0$ and $\det(\Re\left(\rho\right))>0$ for $i\in\{1,2\}$.  $\forall\vec{s}\in\mathbb{C}^2$
\begin{align}\label{Fubi}&\int_{-\infty}^\infty\hspace{-0.3cm}{{\d}x}\;\mathcal{M}\big[\Psi{e}^{-{\rho_{22}}\ln^2(\cdot)}\big]\Big(\frac{s_2}{2}+2(\rho_{22}-\rho_{12})x\Big)e^{-(\rho_{11}+\rho_{22}-2\rho_{12})x^2+\frac{s_1-s_2}{2}x}\\&=\sqrt{\frac{\pi}{\rho_{11}+\rho_{22}-2\rho_{12}}}e^{\frac{(s_1-s_2)^2}{16(\rho_{11}+\rho_{22}-2\rho_{12})}}\mathcal{M}\big[\Psi{e}^{-{\frac{\rho_{11}\rho_{22}-\rho_{12}^2}{\rho_{11}+\rho_{22}-2\rho_{12}}}\ln^2(\cdot)}\big]\left(\frac{s_1\rho_{22}+s_2\rho_{11}-(s_1+s_2)\rho_{12}}{2(\rho_{11}+\rho_{22}-2\rho_{12})}\right)\nonumber\end{align}
\end{proposition}
\begin{proof}[Proof]
We apply Fubini's theorem on 
$\int_{0}^\infty\frac{{\d}t_1}{t_1}\cdots\int_{0}^\infty\frac{{\d}t_d}{t_d}\Psi(\prod_{l=1}^{d}t_l)e^{-\sum_{1\leq{i,j}\leq{d}}\rho_{ij}\ln(t_i)\ln(t_j)}\prod_{l=1}^{d}t_l^{\frac{s_l}{2}}$ in the case $d=2$.
We did not use $\Psi(t)=t^{-\frac{1}{2}}\Psi\left(1/t\right)+\frac{t^{-\frac{1}{2}}-1}{2}$ to obtain this mean value property, only that the expression is absolutely integrable was relevant, {\em i.e.} we can replace $\Psi$ by other test functions. For $\rho_{12}=0$, $\rho_{22}=\gamma$, $\rho_{11}=\rho\gamma/(\gamma-\rho)$ and $s_1=s_2=2s$ formula \ref{Fubi} reduces to
$\mathcal{M}[\Psi{e}^{-\rho\ln^2(\cdot)}]\left(s\right)=\int_{-\infty}^\infty{{\d}q}\;\mathcal{M}\big[\Psi{e}^{-{\gamma}\ln^2(\cdot)}\big]\hspace{-0.05cm}\left(q\right){\exp({-\frac{(q-s)^2}{4(\gamma-\rho)}})}\big/\hspace{-0.05cm}{\sqrt{4{\pi}(\gamma-\rho)}}$,
where we substituted $2\gamma{x}=q$. 
\end{proof}
\begin{lemma}\label{Result3D} Assume $\rho$ satisfies the inequalities $\Re\left(\rho_{ii}\right)>0$ for $i\in\{1,2,3\}$, $\Re\big(\mathcal{R}_{ij}^{\Re(\rho)}\big)>0$ for $i,j\in\{1,2,3\}$ with $i\neq{j}$ and ${\det(\Re(\rho))}>0$. The following identity holds $\forall\vec{s}\in\mathbb{C}^3$
\begin{align}&\Xi\left(\left(\begin{matrix}\rho_{11}&\rho_{12}&\rho_{13}\\\rho_{12}&\rho_{22}&\rho_{23}\\\rho_{13}&\rho_{23}&\rho_{33}\end{matrix}\right),\left(\begin{matrix}s_1\\s_2\\s_3\end{matrix}\right)\right)-\Xi\left(\left(\begin{matrix}\rho_{11}&\rho_{12}&\rho_{13}\\\rho_{12}&\rho_{22}&\rho_{23}\\\rho_{13}&\rho_{23}&\rho_{33}\end{matrix}\right),\left(\begin{matrix}1-s_1\\1-s_2\\1-s_3\end{matrix}\right)\right)\\&=\frac{1}{2^2}\big[\mathrm{e}(\rho,s_1,s_2-1,s_3)-\mathrm{e}(\rho,s_1-1,s_2,s_3-1)+\mathrm{e}(\rho,s_1-1,s_2,s_3)-\mathrm{e}(\rho,s_1-1,s_2-1,s_3)\nonumber\\&\hspace{1.1cm}\mathrm{e}(\rho,s_1-1,s_2-1,s_3-1)-\mathrm{e}(\rho,s_1,s_2,s_3)+\mathrm{e}(\rho,s_1,s_2,s_3-1)-\mathrm{e}(\rho,s_1,s_2-1,s_3-1)\big]\nonumber\\&+\frac{\pi}{2^2}\Big[\mathcal{C}^\rho_3(s_1,s_2,s_3)-\mathcal{C}^\rho_3(s_1,s_2-1,s_3)+\mathcal{C}^\rho_3(s_1-1,s_2-1,s_3)-\mathcal{C}^\rho_3(s_1-1,s_2,s_3)\nonumber\\&\hspace{0.7cm}+\mathcal{C}^\rho_2(s_1,s_3,s_2)-\mathcal{C}^\rho_2(s_1,s_3-1,s_2)+\mathcal{C}^\rho_2(s_1-1,s_3-1,s_2)-\mathcal{C}^\rho_2(s_1-1,s_3,s_2)\nonumber\\&\hspace{0.8cm}+\mathcal{C}^\rho_1(s_2,s_3,s_1)-\mathcal{C}^\rho_1(s_2,s_3-1,s_1)+\mathcal{C}^\rho_1(s_2-1,s_3-1,s_1)-\mathcal{C}^\rho_1(s_2-1,s_3,s_1)\Big]\nonumber\\&+\frac{\pi^{1/2}}{2^2}\Bigg[\frac{1}{\sqrt{\rho_{11}}}\exp\left(\frac{(s_1-1)^2}{16\rho_{11}}\right)\Xi\left(\frac{1}{\rho_{11}}\left(\begin{matrix}\mathcal{R}_{12}^\rho&\mathcal{T}_{231}^{\rho}\\\mathcal{T}_{231}^{\rho}&\mathcal{R}_{13}^\rho\end{matrix}\right),\left(\begin{matrix}1-s_2+(s_1-1)\frac{\rho_{12}}{\rho_{11}}\\1-s_3+(s_1-1)\frac{\rho_{13}}{\rho_{11}}\end{matrix}\right)\right)\nonumber\\&\hspace{1.1cm}-\frac{1}{\sqrt{\rho_{11}}}\exp\left(\frac{s_1^2}{16\rho_{11}}\right)\Xi\left(\frac{1}{\rho_{11}}\left(\begin{matrix}\mathcal{R}_{12}^\rho&\mathcal{T}_{231}^{\rho}\\\mathcal{T}_{231}^{\rho}&\mathcal{R}_{13}^\rho \end{matrix}\right),\left(\begin{matrix}1-s_2+s_1\frac{\rho_{12}}{\rho_{11}}\\1-s_3+s_1\frac{\rho_{13}}{\rho_{11}}\end{matrix}\right)\right)\nonumber\\&\hspace{1.1cm}+\frac{1}{\sqrt{\rho_{22}}}\exp\left(\frac{(s_2-1)^2}{16\rho_{22}}\right)\Xi\left(\frac{1}{\rho_{22}}\left(\begin{matrix}\mathcal{R}_{12}^\rho&\mathcal{T}_{132}^{\rho}\\\mathcal{T}_{132}^{\rho}&\mathcal{R}_{23}^\rho\end{matrix}\right),\left(\begin{matrix}1-s_1+(s_2-1)\frac{\rho_{12}}{\rho_{22}}\\1-s_3+(s_2-1)\frac{\rho_{23}}{\rho_{22}}\end{matrix}\right)\right)\nonumber\\&\hspace{1.1cm}-\frac{1}{\sqrt{\rho_{22}}}\exp\left(\frac{s_2^2}{16\rho_{22}}\right)\Xi\left(\frac{1}{\rho_{22}}\left(\begin{matrix}\mathcal{R}_{12}^\rho&\mathcal{T}_{132}^{\rho}\\\mathcal{T}_{132}^{\rho}&\mathcal{R}_{23}^\rho\end{matrix}\right),\left(\begin{matrix}1-s_1+s_2\frac{\rho_{12}}{\rho_{22}}\\1-s_3+s_2\frac{\rho_{23}}{\rho_{22}}\end{matrix}\right)\right)\nonumber\\&\hspace{1.1cm}+\frac{1}{\sqrt{\rho_{33}}}\exp\left(\frac{(s_3-1)^2}{16\rho_{33}}\right)\Xi\left(\frac{1}{\rho_{33}}\left(\begin{matrix}\mathcal{R}_{13}^\rho&\mathcal{T}_{123}^{\rho}\\\mathcal{T}_{123}^{\rho}&\mathcal{R}_{23}^\rho\end{matrix}\right),\left(\begin{matrix}1-s_1+(s_3-1)\frac{\rho_{13}}{\rho_{33}}\\1-s_2+(s_3-1)\frac{\rho_{23}}{\rho_{33}}\end{matrix}\right)\right)\nonumber\\&\hspace{1.1cm}-\frac{1}{\sqrt{\rho_{33}}}\exp\left(\frac{s_3^2}{16\rho_{33}}\right)\Xi\left(\frac{1}{\rho_{33}}\left(\begin{matrix}\mathcal{R}_{13}^\rho&\mathcal{T}_{123}^{\rho}\\\mathcal{T}_{123}^{\rho}&\mathcal{R}_{23}^\rho\end{matrix}\right),\left(\begin{matrix}1-s_1+s_3\frac{\rho_{13}}{\rho_{33}}\\1-s_2+s_3\frac{\rho_{23}}{\rho_{33}}\end{matrix}\right)\right)\Bigg]\nonumber
\end{align}
where we denote for $k\in\{1,2,3\}$ by $\mathcal{C}_k^\rho$ the combination
$$\mathcal{C}_k^\rho(x,y,z):=\hspace{-0.2cm}\sum_{\substack{{1\leq{i}<j\leq3}\\{i\neq{k}\neq{j}}}}\frac{1}{\sqrt{\mathcal{R}_{ij}^\rho}}\exp\left({\frac{\rho_{jj}x^2+\rho_{ii}y^2-2\rho_{ij}xy}{16\mathcal{R}_{ij}^\rho}}\right)\Xi_{\frac{\det(\rho)}{\mathcal{R}_{ij}^\rho}}\left(1-z-\frac{x\mathcal{T}_{kij}^{\rho}+y\mathcal{T}_{kji}^{\rho}}{\mathcal{R}_{ij}^\rho}\right)$$
\end{lemma}
\begin{proof}[Proof] Let $\Sigma_d$ denotes the group of permutations of $d$ elements and we have
\begin{align}\label{FuEq1}\prod_{i=1}^d\Psi(t_i)=\frac{1}{d!}\sum_{\substack{{0\leq{l}\leq{d}}\\{\sigma\in\Sigma_d}}}\binom{d}{l}\prod_{i=1}^lt_{\sigma(i)}^{-1/2}\Psi\left(\frac{1} {t_{\sigma(i)}}\right)\prod_{j=1}^{d-l}\frac{t_{\sigma(j)}^{-1/2}-1}{2}\end{align}
Explicit the determinant of a symmetric $3\times3$ matrix $\rho$ is
given by the formula
$\det(\rho)=\rho_{11}\rho_{22}\rho_{33}+2\rho_{12}\rho_{13}\rho_{23}-\rho_{11}\rho_{23}^2-\rho_{22}\rho_{13}^2-\rho_{33}\rho_{12}^2$.
Substitution in \ref{2DHelpLemma} implies the expression
\begin{align}\label{2dHelp2}
&\int_{0}^\infty\frac{{\d}t_1}{t_1}\int_{0}^\infty\frac{{\d}t_2}{t_2}\left(\prod_{i=1}^2{t^{{s_i}/2}_i}
\right)e^{-\sum_{1\leq{i,j}=3}\rho_{ij}\ln(t_i)\ln(t_j)}\\&=\frac{\pi{e^{\frac{\rho_{11}s_2^2+\rho_{22}s_1^2-2\rho_{12}s_1s_2}{16(\rho_{11}\rho_{22}-\rho^2_{12})}}}}{\sqrt{\rho_{11}\rho_{22}-\rho^2_{12}}}e^{-\left(\rho_{33}+\frac{2\rho_{12}\rho_{13}\rho_{23}-\rho_{11}\rho_{23}^2-\rho_{22}\rho_{13}^2}{\rho_{11}\rho_{22}-\rho^2_{12}}\right)\ln^2(t_3)}t_3^{\frac{s_1(\rho_{12}\rho_{23}-\rho_{22}\rho_{13})+s_2(\rho_{12}\rho_{13}-\rho_{11}\rho_{23})}{2(\rho_{11}\rho_{22}-\rho^2_{12})}}\nonumber
\end{align}
hence we have
$$\mathrm{e}(\rho,\vec{s})=\frac{\pi^{3/2}}{\sqrt{\det(\rho)}}\exp\Big({\frac{\sum_{1\leq{k}\leq3}s_k^2\sum_{1\leq{i}<j\leq3,i\neq{k}\neq{j}}\mathcal{R}_{ij}^\rho+2\sum_{1\leq{i}<j\leq3}s_is_j\sum_{1\leq{k}\leq3}\mathcal{T}_{ijk}^{\rho}}{16\det(\rho)}}\Big)$$\vspace{-0.15cm}\end{proof}
\begin{lemma}\label{4TermRelation}
Assume $\rho$ satisfies the inequalities $\Re\left(\rho_{ii}\right)>0$ for $i\in\{1,2,3\}$, $\Re\big(\mathcal{R}_{ij}^{\Re(\rho)}\big)>0$ for $i,j\in\{1,2,3\}$ with $i\neq{j}$ and ${\det(\Re(\rho))}>0$. We have
\begin{align}\label{6Term}
&\Xi\left(\rho,\left(\begin{matrix}\frac{1+s_1}{2}\\\frac{1+s_2}{2}\\\frac{1+s_3}{2}\end{matrix}\right)\right)=\Xi\left(\left(\begin{matrix}\rho_{11}&\rho_{12}&-\rho_{13}\\\rho_{12}&\rho_{22}&-\rho_{23}\\-\rho_{13}&-\rho_{23}&\rho_{33}\end{matrix}\right),\left(\begin{matrix}\frac{1+s_1}{2}\\\frac{1+s_2}{2}\\\frac{1-s_3}{2}\end{matrix}\right)\right)\\&\hspace{2cm}+\frac{\sqrt{\pi}}{2}\frac{e^{\frac{1+s_3^2}{64\rho_{33}}}}{\sqrt{\rho_{33}}}\Bigg[e^{\frac{-s_3}{32\rho_{33}}}\Xi\left(\left(\begin{matrix}\rho_{11}-\frac{\rho_{13}^2}{\rho_{33}}&\rho_{12}-\frac{\rho_{13}\rho_{23}}{\rho_{33}}\\\rho_{12}-\frac{\rho_{13}\rho_{23}}{\rho_{33}}&\rho_{22}-\frac{\rho_{23}^2}{\rho_{33}}\end{matrix}\right),\left(\begin{matrix}\frac{1+s_1+(1-s_3)\rho_{13}/\rho_{33}}{2}\\\frac{1+s_2+(1-s_3)\rho_{23}/\rho_{33}}{2}\end{matrix}\right)\right)\nonumber\\&\hspace{4cm}-e^{\frac{s_3}{32\rho_{33}}}\Xi\left(\left(\begin{matrix}\rho_{11}-\frac{\rho_{13}^2}{\rho_{33}}&\rho_{12}-\frac{\rho_{13}\rho_{23}}{\rho_{33}}\\\rho_{12}-\frac{\rho_{13}\rho_{23}}{\rho_{33}}&\rho_{22}-\frac{\rho_{23}^2}{\rho_{33}}\end{matrix}\right),\left(\begin{matrix}\frac{1+s_1-(1+s_3)\rho_{13}/\rho_{33}}{2}\\\frac{1+s_2-(1+s_3)\rho_{23}/\rho_{33}}{2}\end{matrix}\right)\right)\Bigg]\nonumber\\&=\Xi\left(\left(\begin{matrix}\rho_{11}&\rho_{12}&-\rho_{13}\\\rho_{12}&\rho_{22}&-\rho_{23}\\-\rho_{13}&-\rho_{23}&\rho_{33}\end{matrix}\right),\left(\begin{matrix}\frac{1-s_1}{2}\\\frac{1-s_2}{2}\\\frac{1+s_3}{2}\end{matrix}\right)\right)\nonumber\\&+\frac{\sqrt{\pi}}{2}\frac{e^{\frac{1+s_1^2}{64\rho_{11}}}}{\sqrt{\rho_{11}}}\Bigg[e^{\frac{-s_1}{32\rho_{11}}}\Xi\left(\left(\begin{matrix}\rho_{22}-\frac{\rho_{12}^2}{\rho_{11}}&\rho_{23}-\frac{\rho_{12}\rho_{13}}{\rho_{11}}\\\rho_{23}-\frac{\rho_{12}\rho_{13}}{\rho_{11}}&\rho_{33}-\frac{\rho_{13}^2}{\rho_{11}}\end{matrix}\right),\left(\begin{matrix}\frac{1+s_2+(1-s_1)\rho_{12}/\rho_{11}}{2}\\\frac{1+s_3+(1-s_1)\rho_{13}/\rho_{11}}{2}\end{matrix}\right)\right)\nonumber\\&\hspace{2cm}-e^{\frac{s_1}{32\rho_{11}}}\Xi\left(\left(\begin{matrix}\rho_{22}-\frac{\rho_{12}^2}{\rho_{11}}&\rho_{23}-\frac{\rho_{12}\rho_{13}}{\rho_{11}}\\\rho_{23}-\frac{\rho_{12}\rho_{13}}{\rho_{11}}&\rho_{33}-\frac{\rho_{13}^2}{\rho_{11}}\end{matrix}\right),\left(\begin{matrix}\frac{1+s_2-(1+s_1)\rho_{12}/\rho_{11}}{2}\\\frac{1+s_3-(1+s_1)\rho_{13}/\rho_{11}}{2}\end{matrix}\right)\right)\Bigg]\nonumber\\&+\frac{\sqrt{\pi}}{2}\frac{e^{\frac{1+s_2^2}{64\rho_{22}}}}{\sqrt{\rho_{22}}}\Bigg[e^{\frac{-s_2}{32\rho_{22}}}\Xi\left(\left(\begin{matrix}\rho_{11}-\frac{\rho_{12}^2}{\rho_{22}}&-\rho_{13}+\frac{\rho_{12}\rho_{23}}{\rho_{22}}\\-\rho_{13}+\frac{\rho_{12}\rho_{23}}{\rho_{22}}&\rho_{33}-\frac{\rho_{23}^2}{\rho_{22}}\end{matrix}\right),\left(\begin{matrix}\frac{1-s_1-(1-s_2)\rho_{12}/\rho_{22}}{2}\\\frac{1+s_3+(1-s_2)\rho_{23}/\rho_{22}}{2}\end{matrix}\right)\right)\nonumber\\&\hspace{2cm}-e^{\frac{s_2}{32\rho_{22}}}\Xi\left(\left(\begin{matrix}\rho_{11}-\frac{\rho_{12}^2}{\rho_{22}}&-\rho_{13}+\frac{\rho_{12}\rho_{23}}{\rho_{22}}\\-\rho_{13}+\frac{\rho_{12}\rho_{23}}{\rho_{22}}&\rho_{33}-\frac{\rho_{23}^2}{\rho_{22}}\end{matrix}\right),\left(\begin{matrix}\frac{1-s_1+(1+s_2)\rho_{12}/\rho_{22}}{2}\\\frac{1+s_3-(1+s_2)\rho_{23}/\rho_{22}}{2}\end{matrix}\right)\right)\Bigg]\nonumber
\end{align}
\end{lemma}
\begin{proof}[Proof] \ref{6Term} arises by a comparison of two calculations with \ref{SK->1-SK}, a computation corresponding to
$$\Big(\frac{1+s_1}{2},\frac{1+s_2}{2},\frac{1+s_3}{2}\Big)\rightarrow\Big(\frac{1+s_1}{2},\frac{1+s_2}{2},\frac{1-s_3}{2}\Big)$$
and a computation corresponding to the passage
$$\Big(\frac{1+s_1}{2},\frac{1+s_2}{2},\frac{1+s_3}{2}\Big)\rightarrow\Big(\frac{1-s_1}{2},\frac{1+s_2}{2},\frac{1+s_3}{2}\Big)\rightarrow\Big(\frac{1-s_1}{2},\frac{1-s_2}{2},\frac{1+s_3}{2}\Big)$$\end{proof}

\vspace{-0.6cm}\subsection{\textit{Integrals related to $\Xi_{\rho}(s)+\Xi_{\rho}(1-s)$}}\label{3dflip}
\begin{proposition}\label{3dRewritten}
Suppose $\Re(\gamma)>0<\Re(\rho)$ and that the $3\times3$ matrix 
$\rho$ with $\rho_{33}:=\gamma$, $\rho_{11}=\rho_{22}=\rho+{s}^2\gamma$ and $\rho_{13}=\rho_{23}=s\gamma$, $\rho_{12}=-s^2\gamma$
satisfies $\Re\left(\rho_{ii}\right)>0$ for $i\in\{1,2,3\}$, $\mathcal{R}_{ij}^{\Re(\rho)}>0$ and ${\det(\Re(\rho))}>0$ for $i,j\in\{1,2,3\}$ with $i\neq{j}$. We have
\begin{align}\label{s->1-sStep4}
&\Xi\left(\left(\begin{matrix}\rho+s^2\gamma&s^2\gamma&s\gamma\\s^2\gamma&\rho+s^2\gamma&s\gamma\\s\gamma&s\gamma&\gamma\end{matrix}\right),\left(\begin{matrix}\frac{1}{2}\\\frac{1}{2}\\\frac{1}{2}\end{matrix}\right)\right)=\Xi\left(\left(\begin{matrix}\rho+s^2\gamma&s^2\gamma&-s\gamma\\s^2\gamma&\rho+s^2\gamma&-s\gamma\\-s\gamma&-s\gamma&\gamma\end{matrix}\right),\left(\begin{matrix}\frac{1}{2}\\\frac{1}{2}\\\frac{1}{2}\end{matrix}\right)\right)\end{align}
\begin{align}\label{s->1-sStep7}
&\Xi\left(\left(\begin{matrix}\frac{\rho^2+2\rho\gamma{s}^2}{\rho+\gamma{s}^2}&\frac{\rho\gamma{s}}{\rho+\gamma{s}^2}\\\frac{\rho\gamma{s}}{\rho+\gamma{s}^2}&\frac{\rho\gamma}{\rho+\gamma{s}^2}\end{matrix}\right),\left(\begin{matrix}\frac{\rho+2\gamma{s}^2}{2(\rho+\gamma{s}^2)}\\\frac{\rho+\gamma{s}^2+\gamma{s}}{2(\rho+\gamma{s}^2)}\end{matrix}\right)\right)-\Xi\left(\left(\begin{matrix}\frac{\rho^2+2\rho\gamma{s}^2}{\rho+\gamma{s}^2}&\frac{\rho\gamma{s}}{\rho+\gamma{s}^2}\\\frac{\rho\gamma{s}}{\rho+\gamma{s}^2}&\frac{\rho\gamma}{\rho+\gamma{s}^2}\end{matrix}\right),\left(\begin{matrix}\frac{\rho}{2(\rho+\gamma{s}^2)}\\\frac{\rho+\gamma{s}^2-\gamma{s}}{2(\rho+\gamma{s}^2)}\end{matrix}\right)\right)\nonumber\\=&\Xi\left(\left(\begin{matrix}\frac{\rho^2+2\rho\gamma{s}^2}{\rho+\gamma{s}^2}&-\frac{\rho\gamma{s}}{\rho+\gamma{s}^2}\\-\frac{\rho\gamma{s}}{\rho+\gamma{s}^2}&\frac{\rho\gamma}{\rho+\gamma{s}^2}\end{matrix}\right),\left(\begin{matrix}\frac{\rho+2\gamma{s}^2}{2(\rho+\gamma{s}^2)}\\\frac{\rho+\gamma{s}^2-\gamma{s}}{2(\rho+\gamma{s}^2)}\end{matrix}\right)\right)-\Xi\left(\left(\begin{matrix}\frac{\rho^2+2\rho\gamma{s}^2}{\rho+\gamma{s}^2}&-\frac{\rho\gamma{s}}{\rho+\gamma{s}^2}\\-\frac{\rho\gamma{s}}{\rho+\gamma{s}^2}&\frac{\rho\gamma}{\rho+\gamma{s}^2}\end{matrix}\right),\left(\begin{matrix}\frac{\rho}{2(\rho+\gamma{s}^2)}\\\frac{\rho+\gamma{s}^2+\gamma{s}}{2(\rho+\gamma{s}^2)}\end{matrix}\right)\right)
\end{align}
if and only if
$\Xi^2_{\rho}\left(\frac{1+s}{2}\right)-\Xi^2_{\rho}\left(\frac{1-s}{2}\right)=0$.
\end{proposition}
\begin{proof}[Proof]For the first equivalence we use \ref{SK->1-SK}, for the second formula \ref{s->1-sStep7} we used \ref{6Term}.\end{proof}
\begin{proposition}\label{2dRewritten}
Let $\Re(\rho+\alpha{s}^2)>0<\Re(\alpha)$ and  $\Re(\alpha{s})<\sqrt{\Re(\alpha)\Re(\rho+\alpha{s}^2)}$. $\forall{n}\in\mathbb{Z}$:\begin{align}\label{1+2n}&\Xi\left(\hspace{-0.05cm}\left(\begin{matrix}\rho+\alpha{s}^2&\alpha{s}\\\alpha{s}&\alpha\end{matrix}\right),\left(\begin{matrix}\frac{1-16\I\pi(1+2{n})\alpha{s}}{2}\\\frac{1-16\I\pi(1+2{n})\alpha}{2}\end{matrix}\right)\hspace{-0.05cm}\right)=\Xi\left(\hspace{-0.05cm}\left(\begin{matrix}\rho+\alpha{s}^2&-\alpha{s}\\-\alpha{s}&\alpha\end{matrix}\right),\left(\begin{matrix}\frac{1-16\I\pi(1+2{n})\alpha{s}}{2}\\\frac{1+16\I\pi(1+2{n})\alpha}{2}\end{matrix}\right)\hspace{-0.05cm}\right)\end{align}
if and only if
$\Xi_{\rho}\left(\frac{1+s}{2}\right)+\Xi_{\rho}\left(\frac{1-s}{2}\right)=0$.
\end{proposition}
Observe with the M\"obius-transformations
$$\alpha(\gamma)=\gamma\rho/(\rho+\gamma{s^2})\quad\quad\text{and}\quad\quad\gamma(\alpha)=\alpha\rho/(\rho-\alpha{s^2})$$
we can transform from the $2\times2$ matrices appearing in \ref{1+2n} and \ref{3dRewritten} respectively, the convenient second choice yields that $\Xi^2_{\rho}\left(\frac{1+s}{2}\right)-\Xi^2_{\rho}\left(\frac{1-s}{2}\right)=0$ if and only if
\begin{align}\label{s->1-sStepRewrite}
&\sum_{i=0}^1(-1)^i\Xi\left(\left(\begin{matrix}\rho+\alpha{s}^2&\alpha{s}\\\alpha{s}&\alpha\end{matrix}\right),\left(\begin{matrix}\frac{1+(-1)^i\frac{\alpha}{\rho}{s}^2}{2}\\\frac{1+(-1)^i\frac{\alpha}{\rho}{s}}{2}\end{matrix}\right)\right)=\sum_{i=0}^1(-1)^i\Xi\left(\left(\begin{matrix}\rho+\alpha{s}^2&-\alpha{s}\\-\alpha{s}&\alpha\end{matrix}\right),\left(\begin{matrix}\frac{1+(-1)^i\frac{\alpha}{\rho}{s}^2}{2}\\\frac{1-(-1)^i\frac{\alpha}{\rho}{s}}{2}\end{matrix}\right)\right)\end{align}\vspace{-0.4cm}

Notice the arguments of the integrals in \ref{2dRewritten} and \ref{s->1-sStepRewrite} agree for ${s}=16\I\pi\rho(1+2\mathbb{Z})$ in this case we have \ref{2dRewritten} $\Rightarrow$ \ref{s->1-sStepRewrite}. If moreover $\Xi_{\rho}\left(\frac{1+s}{2}\right)+\Xi_{\rho}\left(\frac{1-s}{2}\right)=0\Rightarrow{s}\in16\I\pi\rho(1+2\mathbb{Z})$ would hold the zeros of $\Xi_{\rho}(s)+\Xi_{\rho}(1-s)$ and $\Xi_{\rho}(s)-\Xi_{\rho}(1-s)$ would be interlacing, some evidence for the strategy that the Hermite-Biehler theorem could get used to show $\Xi_{\rho}(s)=0\Rightarrow\Re(s)<1/2$.\vspace{-0.2cm}
\section{\textit{The operators $\id+4t\partial_t$ and $2t^2\partial_t^2+3t\partial_t$}}\label{Jensen}
Consider the operators
$(H_\alpha\Psi)(t):=\left(\Psi+\alpha{t}\Psi'\right)(t)$.
The function $(H_4\Psi)(t)$
satisfies a functional equation similar to \ref{FuEq1}, for instance
$(H_4\Psi)(t)=-{t}^{-1/2}(H_4\Psi)(1/t)-({t}^{-1/2}+1)/{2}$
and clearly the previous calculations are with some sign changes also in some sense valid for the family of integrals where we partially replace $\Psi$ by $H_4\Psi$.  Also observe for appropriate test functions $\psi$
\begin{align}\label{hConect}&\mathcal{M}\big[(H_\alpha\psi){e}^{-\rho\ln^2(t)}\big](s/2)=\left(1-\frac{\alpha}{2}\id\right)\mathcal{M}\big[\psi{e}^{-\rho\ln^2(t)}\big](s/2)+4\alpha\rho\partial_s\mathcal{M}\big[\psi{e}^{-\rho\ln^2(t)}\big](s/2)\end{align}
holds, hence 
$\mathcal{M}\big[(H_\alpha\Psi)\exp({-\rho\ln^2(t)})\big](s/2)=(1-\frac{\alpha}{2}\id)\Xi_\rho(s)+4\alpha\rho\partial_s\Xi_\rho(s)$.

Let us here briefly comment on a familiar class of integrals where we substitute in \ref{EquiSym4} the function $\Psi(t)$ by the Jensen function $\left(2t^2\Psi''+3t\Psi'\right)(t)$ for instance
\vspace{-0.15cm}\begin{align}\label{xi}&\xi(\rho,\vec{s}\;):=\int_{0}^\infty\frac{{\d}t_1}{t_1}\cdots\int_{0}^\infty\frac{{\d}t_d}{t_d}\left(\prod_{i=1}^d{t^{{s_i}/2}_i}(\Delta_4\Psi)(t_i)
\right)e^{-\sum_{1\leq{i,j}\leq{d}}\rho_{ij}\ln(t_i)\ln(t_j)}\end{align}
where $\Delta_\alpha:=H_\alpha^2-\id$, for instance
$(\Delta_4\Psi)(t):=8\left(2t^2\Psi''+3t\Psi'\right)(t)$.
If the integrals \ref{xi} are convergent the conjugation
$\overline{\xi(\rho,\vec{s})}=\xi{(\overline{\rho}},\overline{\vec{s}})$ is valid and we have a tedious symmetry:
\vspace{-0.1cm}\begin{align}\label{xi:s->1-s2}&\xi\left(\rho,\vec{s}\;\right)=\xi\left(\left(\begin{matrix}\rho_{11}&.&-\rho_{1k}&.&\rho_{1d}\\.&.&.&.&.\\-\rho_{1k}&.&\rho_{kk}&.&-\rho_{kd}\\.&.&.&.&.\\\rho_{1d}&.&-\rho_{kd}&.&\rho_{dd}\end{matrix}\right),\left(\begin{matrix}s_1\\.\\1-s_k\\.\\s_d\end{matrix}\right)\right)\end{align}where the matrix on the r.h.s is obtained from $\rho$ by acting on $\rho_{ik}$ by a sign if $i\neq{k}$ and acting on all other matrix entries by the identity: 
The  function $\Delta_4\Psi$ is one of the motivations for the work of P\'olya, Brujin and Newman \cite{Cs}, we have
$(\Delta_4\Psi)(\cdot)=t^{-\frac{1}{2}}(\Delta_4\Psi)\left({1}/{t}\right)$.
 This shows that $\Delta_4\Psi$ does decay rapidly at $\infty$ and also at $0$ while $\Psi(t)$ has a $t^{-1/2}/2$ singularity. Also observe that iteration of \ref{hConect} yields
the identification
\begin{align}\label{SecondOrder}\mathcal{M}\big[(\Delta_4\Psi)e^{-\rho\ln^2(t)}\big](s/2)=4s(s-1-8\rho)\Xi_\rho(s)+32\rho(1-2s)\partial_s\Xi_\rho(s)+(16\rho)^2\partial^2_s\Xi_\rho(s)\end{align}

The first-order differential equation \ref{hConect} essentially descends from the heat equation and can be ``solved" by a standard method. To ``solve" \ref{hConect} will also help to tackle the more involved second order equation \ref{SecondOrder}:
\begin{proposition}\label{fde1} For $\alpha\in\mathbb{C}\setminus0$ and $z\in\mathbb{C}$ we have
\begin{align*}&e^{\frac{-s^2+\frac{4}{\alpha}{s}}{16\rho}}\mathcal{M}\big[\Psi{e}^{-\rho\ln^2(\cdot)}\big]\left(\frac{s}{2}\right)\\&=e^{\frac{-z^2+\frac{4}{\alpha}{z}}{16\rho}}\mathcal{M}\big[\Psi{e}^{-\rho\ln^2(\cdot)}\big]\left(\frac{z}{2}\right)+\frac{1}{4\alpha\rho}\int_z^s\hspace{-0.15cm}\d{t}\;{e}^{\frac{-t^2+\frac{4}{\alpha}{t}}{16\rho}}\mathcal{M}\big[{e}^{-\rho\ln^2(\cdot)}H_\alpha\Psi\big]\left(\frac{t}{2}\right)\\&=e^{\frac{-z^2+\frac{4}{\alpha}{z}}{16\rho}}\mathcal{M}\big[\Psi{e}^{-\rho\ln^2(\cdot)}\big]\left(\frac{z}{2}\right)+\frac{1}{2}\sqrt{\frac{\pi}{\rho}}\left(e^{\frac{1-\frac{4}{\alpha}(\frac{\alpha}{2}-1)x}{16\rho}}-e^{\frac{\frac{4}{\alpha}{x}}{16\rho}}\right)\hspace{-0.1cm}\Big\vert^{s}_{z}\\&+{e}^{\frac{\frac{4}{\alpha}-1}{16\rho}}\frac{\frac{\alpha}{2}-1}{4\alpha\rho}\int_{1-z}^{1-s}\hspace{-0.15cm}\d{t}\;{e}^{\frac{-t^2+\frac{4}{\frac{\alpha}{\frac{\alpha}{2}-1}}{t}}{16\rho}}\mathcal{M}\big[{e}^{-\rho\ln^2(\cdot)}H_{\frac{\alpha}{\frac{\alpha}{2}-1}}\Psi\big]\left(\frac{t}{2}\right)\end{align*}
\end{proposition}
\begin{proof}[Proof] The first equals just assumes that $\Psi$ is well behaved for integration and relies on \ref{hConect} and holds for arbitrary test functions. The second equals used $\Psi(t)=t^{-\frac{1}{2}}\Psi\left(1/t\right)+(t^{-\frac{1}{2}}-1)/2$. \end{proof}
\vspace{-0.3cm}\begin{korollar}\label{HalphaVanishing} If ${e}^{\frac{-z^2+\frac{4}{\alpha}{z}}{16\rho}}\Xi_\rho(z)={e}^{\frac{-z'^2+\frac{4}{\alpha}{z'}}{16\rho}}\Xi_\rho(z')$ for $\alpha\in\mathbb{C}\hspace{-0.05cm}\setminus\hspace{-0.05cm}0$ then we have the vanishing
$\int_z^{z'}\hspace{-0.1cm}\d{t}\;{e}^{\frac{-t^2+\frac{4}{\alpha}{t}}{16\rho}}\hspace{-0.1cm}\mathcal{M}\big[{e}^{-\rho\ln^2(\cdot)}H_\alpha\Psi\big]\hspace{-0.1cm}\left(\frac{t}{2}\right)=0$.\end{korollar}


The proposition \ref{SequendDiff} seems more convenient but also the procedure \ref{fde1} can be iterated, for example if $\alpha\in\mathbb{C}\setminus0$  and $z_i\in\mathbb{C}$ with $0=\mathcal{M}\big[(H^{i-1}_\alpha\Psi){e}^{-\rho\ln^2(\cdot)}\big]\left(\frac{z_i}{2}\right)$ for $1\leq{i}\leq{n}$ we find
\begin{align*}\Xi_\rho\left(s\right)=\frac{1}{(4\alpha\rho)^n}e^{\frac{s^2-\frac{4}{\alpha}{s}}{16\rho}}&\int_{z_1}^s\hspace{-0.15cm}\d{t_1}\int_{z_2}^{t_1}\hspace{-0.25cm}\d{t_2}\cdots\int_{z_n}^{t_{n-1}}\hspace{-0.55cm}\d{t_n}\;{e}^{\frac{-t_n^2+\frac{4}{\alpha}{t_n}}{16\rho}}\mathcal{M}\big[{e}^{-\rho\ln^2(\cdot)}H^{n}_\alpha\Psi\big]\left(\frac{t_n}{2}\right)\end{align*}
\begin{proposition}\label{SequendDiff} For $\alpha\in\mathbb{C}\setminus0$ and $z\in\mathbb{C}$ we have
\begin{align}\label{2Diff} &\left(\id-(4\alpha\rho)^2\frac{\partial^2}{\partial{s}^2}\right)\left(e^{\frac{-s^2+\frac{4}{\alpha}{s}}{16\rho}}\mathcal{M}\big[\Psi{e}^{-\rho\ln^2(\cdot)}\big]\left(\frac{s}{2}\right)\right)=e^{\frac{-s^2+\frac{4}{\alpha}{s}}{16\rho}}\mathcal{M}\big[{e}^{-\rho\ln^2(\cdot)}\Delta_\alpha\Psi\big]\left(\frac{s}{2}\right)\end{align}
\end{proposition}
\begin{proof}[Proof] We have $4\alpha\rho\partial_s\Big(e^{\frac{-s^2+\frac{4}{\alpha}{s}}{16\rho}}\mathcal{M}\big[\Psi{e}^{-\rho\ln^2(\cdot)}\big]\left(\frac{s}{2}\right)\Big)={e}^{\frac{-s^2+\frac{4}{\alpha}{s}}{16\rho}}\mathcal{M}\big[{e}^{-\rho\ln^2(\cdot)}H_\alpha\Psi\big]\left(\frac{s}{2}\right)$ by \ref{fde1}, then we iterate the result.\end{proof}
\vspace{-0.05cm}A convenient fundamental system of solutions $\varphi$ of the corresponding homogenous equation
$\left(\id-(4\alpha\rho)^2\frac{\partial^2}{\partial{s}^2}\right)\varphi=0$
are the translated hyperbolic functions $\sinh\big(\frac{\beta_1-s}{4\alpha\rho}\big)$ and $\cosh\big(\frac{\beta_2-s}{4\alpha\rho}\big)$ where $\beta_1\notin\beta_2+\pi\I4\alpha\rho(\frac{1}{2}+\mathbb{Z})$. The general solution of the non-homogeneous linear differential equation \ref{SequendDiff} can be obtained with a method of Lagrange, known as variation of parameters:
\begin{theorem}\label{SequendDiffSol} For $\alpha\in\mathbb{C}\setminus0$, $z\in\mathbb{C}$ and  $\vec{\beta}\in\mathbb{C}^2$ with $\beta_1\notin\beta_2+\pi\I4\alpha\rho(\frac{1}{2}+\mathbb{Z})$ we have 
\begin{align}\label{2DiffSol} e^{\frac{-s^2+\frac{4}{\alpha}{s}}{16\rho}}\mathcal{M}\big[\Psi{e}^{-\rho\ln^2(\cdot)}\big]\left(\frac{s}{2}\right)=&A_\rho(\vec{\beta},\alpha,z)\sinh\left(\frac{\beta_1-s}{4\alpha\rho}\right)+B_\rho(\vec{\beta},\alpha,z)\cosh\left(\frac{\beta_2-s}{4\alpha\rho}\right)\\&+\frac{1}{4\alpha\rho}\int_{z}^{s}\hspace{-0.2cm}\d{t}\;\sinh\left(\frac{s-t}{4\alpha\rho}\right)e^{\frac{-t^2+\frac{4}{\alpha}{t}}{16\rho}}\mathcal{M}\big[{e}^{-\rho\ln^2(\cdot)}\Delta_\alpha\Psi\big]\left(\frac{t}{2}\right)\nonumber\end{align}
where $A_\rho(\vec{\beta},\alpha,z)$ and $B_\rho(\vec{\beta},\alpha,z)$ are given by
\begin{align}\label{Constants} &\left(\hspace{-0.05cm}\begin{matrix}A_\rho\big(\vec{\beta},\alpha,z\big)\\B_\rho\big(\vec{\beta},\alpha,z\big)\end{matrix}\hspace{-0.05cm}\right)\hspace{-0.05cm}=\hspace{-0.05cm}\frac{{e}^{\frac{-z^2+\frac{4}{\alpha}{z}}{16\rho}}}{\cosh\left(\frac{\beta_2-\beta_1}{4\alpha\rho}\right)}\hspace{-0.1cm}\left(\hspace{-0.05cm}\begin{matrix}\sinh\left(\frac{\beta_2-z}{4\alpha\rho}\right)&\hspace{-0.3cm}-\cosh\left(\frac{\beta_2-z}{4\alpha\rho}\right)\\-\cosh\left(\frac{\beta_1-z}{4\alpha\rho}\right)&\hspace{-0.3cm}\sinh\left(\frac{\beta_1-z}{4\alpha\rho}\right)\end{matrix}\hspace{-0.05cm}\right)\hspace{-0.15cm}\left(\hspace{-0.05cm}\begin{matrix}-\mathcal{M}\big[{e}^{-\rho\ln^2(\cdot)}\Psi\big]\big(\frac{z}{2}\big)\\\mathcal{M}\big[{e}^{-\rho\ln^2(\cdot)}H_\alpha\Psi\big]\big(\frac{z}{2}\big)\end{matrix}\hspace{-0.05cm}\right)\end{align}
\end{theorem}
We seek to close the system with \ref{Para1} instead of \ref{Constants}, {\em i.e.} we compute the skew-symmetric part of $\mathcal{M}\big[\Psi{e}^{-\rho\ln^2(\cdot)}\big]\left(\frac{s}{2}\right)$ with respect to the critical line $\Re(s)=1/2$ and to do this it is practical to choose $\alpha=4$: We have
$(\Delta_\alpha\Psi)(t)=t^{-\frac{1}{2}}\left(\Delta_\alpha\Psi\left({1}/{t}\right)+\frac{\alpha(\alpha-4)}{4}\left(H_4\Psi\left({1}/{t}\right)+\frac{1}{2}\right)\right)$,
hence
\vspace{-0.25cm}\begin{align}\label{MSym}\mathcal{M}\big[{e}^{-\rho\ln^2(\cdot)}\Delta_\alpha\Psi\big]\left(\frac{s}{2}\right)=&\mathcal{M}\big[{e}^{-\rho\ln^2(\cdot)}\Delta_{\alpha}\Psi\big]\left(\frac{1-s}{2}\right)\\&+\frac{\alpha(\alpha-4)}{4}\left(\mathcal{M}\big[{e}^{-\rho\ln^2(\cdot)}H_4\Psi\big]\left(\frac{1-s}{2}\right)+\frac{1}{2}\sqrt{\frac{\pi}{\rho}}e^{\frac{(s-1)^2}{16\rho}}\right)\nonumber\end{align}
We define $\chi_\rho(\varphi,\alpha,s,z):=\int_{z}^{s}\d{t}\;e^{\frac{-t^2+\varphi{t}}{16\rho}}\mathcal{M}\big[{e}^{-\rho\ln^2(\cdot)}\Delta_\alpha\Psi\big]\left(\frac{t}{2}\right)$ and have
\begin{align*} \int_{z}^{s}\hspace{-0.2cm}\d{t}\;\sinh\left(\frac{s-t}{4\alpha\rho}\right)e^{\frac{-t^2+\frac{4}{\alpha}{t}}{16\rho}}\mathcal{M}\big[{e}^{-\rho\ln^2(\cdot)}\Delta_\alpha\Psi\big]\left(\frac{t}{2}\right)=e^{\frac{s}{4\alpha\rho}}\chi_\rho(0,\alpha,s,z)-e^{\frac{-s}{4\alpha\rho}}\chi_\rho\left(\frac{8}{\alpha},\alpha,s,z\right)\end{align*}
Clearly we have $\chi_\rho(\varphi,\alpha,s,z)=-\chi_\rho(\varphi,\alpha,z,s)$ and the transitivity $\chi_\rho(\varphi,\alpha,s,y)+\chi_\rho(\varphi,\alpha,y,z)=\chi_\rho(\varphi,\alpha,s,z)$.
With \ref{MSym} we obtain the transformation law
\vspace{-0.3cm}\begin{align}\label{AlphaTrans1}\chi_\rho(\varphi,\alpha,s,z)=-e^{\frac{\varphi-1}{16\rho}}\Bigg[&\chi_\rho(2-\varphi,\alpha,1-s,1-z)\\&+\frac{\alpha(\alpha-4)}{4}\Bigg(\int_{1-z}^{1-s}\hspace{-0.4cm}\d{t}\;e^{\frac{-t^2+(2-\varphi){t}}{16\rho}}\mathcal{M}\big[{e}^{-\rho\ln^2(\cdot)}H_4\Psi\big]\left(\frac{t}{2}\right)\nonumber\\&\hspace{2.3cm}+\frac{1}{2}\sqrt{\frac{\pi}{\rho}}\begin{cases}\frac{16\rho}{(2-\varphi)}e^{\frac{(2-\varphi)x}{16\rho}}\;\text{if}\;\varphi\neq2\\x\hspace{1.7cm}\;\text{if}\;\varphi=2\end{cases}\Big\vert^{1-s}_{1-z}\Bigg)\Bigg]\nonumber\end{align}
and \ref{fde1} implies
$\chi_\rho(1,\alpha,s,z)+\chi_\rho(1,\alpha,1-s,1-z)=16\rho\frac{\alpha(4-\alpha)}{4}\Big({e}^{\frac{-x^2+{x}}{16\rho}}\Xi_\rho(1-x)-\frac{1}{2}\sqrt{\frac{\pi}{\rho}}e^{\frac{x}{16\rho}}\Big)\Big\vert^{1-s}_{1-z}$.
\begin{theorem}\label{SequendDiffSolFin} We have
\vspace{-0.1cm}\begin{align}\label{2DiffSol} e^{\frac{-s^2+{s}}{16\rho}}\mathcal{M}\big[\Psi{e}^{-\rho\ln^2(\cdot)}\big]\left(\frac{s}{2}\right)=&\frac{e^{\frac{1}{32\rho}}}{4}\sqrt{\frac{\pi}{\rho}}\sinh\left(\frac{\frac{1}{2}-s}{16\rho}\right)+e^{\frac{1}{64\rho}}\mathcal{M}\big[\Psi{e}^{-\rho\ln^2(\cdot)}\big]\left(\frac{1}{4}\right)\cosh\left(\frac{\frac{1}{2}-s}{16\rho}\right)\nonumber\\&+\frac{1}{16\rho}\int_{1/2}^{s}\hspace{-0.2cm}\d{t}\;\sinh\left(\frac{s-t}{16\rho}\right)e^{\frac{-t^2+{t}}{16\rho}}\mathcal{M}\big[{e}^{-\rho\ln^2(\cdot)}\Delta_4\Psi\big]\left(\frac{t}{2}\right)\end{align}
\end{theorem}
\begin{proof}[Proof] In particular for $\alpha=4$ formula \ref{AlphaTrans1} reduces to
\begin{align}\label{LaplaSym}
&\int_{z}^{s}\hspace{-0.25cm}\d{t}\;\sinh\left(\hspace{-0.05cm}\frac{s-t}{16\rho}\hspace{-0.05cm}\right)\hspace{-0.05cm}e^{\frac{-t^2+{t}}{16\rho}}\hspace{-0.1cm}\mathcal{M}\big[{e}^{-\rho\ln^2(\cdot)}\Delta_4\Psi\big]\hspace{-0.1cm}\left(\frac{t}{2}\right)\hspace{-0.1cm}=\hspace{-0.15cm}\int_{1-z}^{1-s}\hspace{-0.2cm}\hspace{-0.4cm}\d{t}\;\sinh\left(\hspace{-0.05cm}\frac{1-s-t}{16\rho}\hspace{-0.05cm}\right)\hspace{-0.05cm}e^{\frac{-t^2+{t}}{16\rho}}\hspace{-0.1cm}\mathcal{M}\big[{e}^{-\rho\ln^2(\cdot)}\Delta_4\Psi\big]\hspace{-0.1cm}\left(\frac{t}{2}\right)
\end{align}
\vspace{-0.7cm}

Hence with \ref{Para1}  we find
\begin{align}\label{NewB} \sqrt{\frac{\pi}{\rho}}\frac{e^{\frac{1-s}{16\rho}}-e^{\frac{s}{16\rho}}}{2}=&A_\rho(\vec{\beta},4,z)\sinh\left(\frac{\beta_1-s}{16\rho}\right)+B_\rho(\vec{\beta},4,z)\cosh\left(\frac{\beta_2-s}{16\rho}\right)\\&-A_\rho(\vec{\beta}',4,1-z)\sinh\left(\frac{\beta_1'-1+s}{16\rho}\right)-B_\rho(\vec{\beta}',4,1-z)\cosh\left(\frac{\beta_2'-1+s}{16\rho}\right)\nonumber\end{align}
$$\Leftrightarrow$$\vspace{-0.7cm}
\begin{align*}A_\rho\hspace{-0.05cm}\left(\vec{\beta},4,z\right)e^{\frac{\beta_1-1}{16\rho}}\hspace{-0.05cm}+\hspace{-0.05cm}A_\rho\hspace{-0.05cm}\left(\vec{\beta}',4,1-z\right)e^{\frac{-\beta_1'}{16\rho}}\hspace{-0.05cm}+\hspace{-0.05cm}B_\rho\hspace{-0.05cm}\left(\vec{\beta},4,z\right)e^{\frac{\beta_2-1}{16\rho}}\hspace{-0.05cm}-\hspace{-0.05cm}B_\rho\hspace{-0.05cm}\left(\vec{\beta}',4,1-z\right)e^{\frac{-\beta_2'}{16\rho}}=\sqrt{\frac{\pi}{\rho}}\hspace{-0.1cm}\Bigg/\hspace{-0.1cm}2\end{align*}\vspace{-0.05cm}In particular if we choose the values $\beta=\beta'$ and $z=1/2$ the previous equation is equivalent to the constraint
$$A_\rho\big(\vec{\beta},4,1/2\big)\big(e^{\frac{\beta_1-1}{16\rho}}+e^{\frac{-\beta_1}{16\rho}}\big)+B_\rho\big(\vec{\beta},4,1/2\big)\big(e^{\frac{\beta_2-1}{16\rho}}-e^{\frac{-\beta_2}{16\rho}}\big)=\sqrt{{\pi}/{\rho}}/2$$

We are also forced by \ref{SequendDiff} to the constraint
$$e^{\frac{1}{64\rho}}\mathcal{M}\big[\Psi{e}^{-\rho\ln^2(\cdot)}\big]\left(\frac{1}{4}\right)=e^{\frac{1}{32\rho}}\left(A_\rho(\vec{\beta},4,1/2)\left(e^{\frac{\beta_1-1}{16\rho}}-e^{\frac{-\beta_1}{16\rho}}\right)+B_\rho(\vec{\beta},4,1/2)\left(e^{\frac{\beta_2-1}{16\rho}}+e^{\frac{-\beta_2}{16\rho}}\right)\right)$$
The two previous constraints now determine $A_\rho(\vec{\beta},4,1/2)$ and $B_\rho(\vec{\beta},4,1/2)$, in particular we obtain the normalisation in the theorem for $\vec{\beta}=(1/2,1/2)$. 
 
 Also observe we have the inequality $\Xi_\rho(r)=\mathcal{M}\big[\Psi{e}^{-\rho\ln^2(\cdot)}\big]\left(\frac{r}{2}\right)>0$ for $r\in\mathbb{R}$.\end{proof}
\vspace{-0.1cm}
Analogous to \ref{SequendDiffSolFin} we obtain for $\tilde{\Xi}_{\rho}(s):=\mathcal{M}\big[{e}^{-\rho\ln^2(\cdot)}H_4\Psi\big]\left(\frac{s}{2}\right)$ the formula
\begin{align*}e^{\frac{-s^2+{s}}{16\rho}}\tilde{\Xi}_{\rho}(s)=&C_\rho\sinh\left(\frac{\frac{1}{2}-s}{16\rho}\right)-\frac{e^{\frac{1}{32\rho}}}{4}\sqrt{\frac{\pi}{\rho}}\cosh\left(\frac{\frac{1}{2}-s}{16\rho}\right)\\&+\frac{1}{16\rho}\int_{1/2}^{s}\d{t}\;\sinh\left(\frac{s-t}{16\rho}\right)e^{\frac{-t^2+{t}}{16\rho}}\mathcal{M}\big[{e}^{-\rho\ln^2(t)}\Delta_4H_4\Psi\big]\left(\frac{t}{2}\right)\end{align*}
where the remaining integral is skew-symmetric with respect to $s\rightarrow1-s$ and we have $(\Delta_4H_4\Psi)=16(4t^3\partial_t^3\Psi+15t^2\partial_t^2\Psi+7t\partial_t\Psi)(t)$. The parameter $C_\rho$ is determined by evaluating at some point $\neq1/2+16\rho\pi\I\mathbb{Z}$ and it is easier calculate with $4\alpha\rho\partial_s\big(e^{\frac{-s^2+\frac{4}{\alpha}{s}}{16\rho}}\mathcal{M}\big[\Psi{e}^{-\rho\ln^2(\cdot)}\big]\left(\frac{s}{2}\right)\big)={e}^{\frac{-s^2+\frac{4}{\alpha}{s}}{16\rho}}\mathcal{M}\big[{e}^{-\rho\ln^2(\cdot)}H_\alpha\Psi\big]\left(\frac{s}{2}\right)$ and identify $\exp({\frac{-s^2+{s}}{16\rho}})\tilde{\Xi}_{\rho}(s)$ with
\begin{align}\label{H4DiffSol2} =&-\frac{e^{\frac{1}{32\rho}}}{4}\sqrt{\frac{\pi}{\rho}}\cosh\left(\frac{\frac{1}{2}-s}{16\rho}\right)-e^{\frac{1}{64\rho}}\mathcal{M}\big[\Psi{e}^{-\rho\ln^2(\cdot)}\big]\left(\frac{1}{4}\right)\sinh\left(\frac{\frac{1}{2}-s}{16\rho}\right)\nonumber\\&+\frac{1}{16\rho}\int_{1/2}^{s}\hspace{-0.2cm}\d{t}\;\cosh\left(\frac{s-t}{16\rho}\right)e^{\frac{-t^2+{t}}{16\rho}}\mathcal{M}\big[{e}^{-\rho\ln^2(\cdot)}\Delta_4\Psi\big]\left(\frac{t}{2}\right)\end{align}
The addition theorems
$\sinh\Big(\frac{s-t}{16\rho}\Big)=-\sinh\Big(\frac{\frac{1}{2}-s}{16\rho}\Big)\cosh\Big(\frac{\frac{1}{2}-t}{16\rho}\Big)+\sinh\Big(\frac{\frac{1}{2}-t}{16\rho}\Big)\cosh\Big(\frac{\frac{1}{2}-s}{16\rho}\Big)$ and $\cosh\Big(\frac{s-t}{16\rho}\Big)=\cosh\Big(\frac{\frac{1}{2}-s}{16\rho}\Big)\cosh\Big(\frac{\frac{1}{2}-t}{16\rho}\Big)-\sinh\Big(\frac{\frac{1}{2}-t}{16\rho}\Big)\sinh\Big(\frac{\frac{1}{2}-s}{16\rho}\Big)$
allow to slightly rewrite \ref{SequendDiffSolFin} and \ref{H4DiffSol2} in
\begin{align*} &e^{\frac{-s^2+{s}}{16\rho}}\Xi_\rho(s)=\Big[\frac{e^{\frac{1}{32\rho}}}{4}\sqrt{\frac{\pi}{\rho}}-\mathrm{a}^+_\rho(s)\Big]\sinh\left(\frac{\frac{1}{2}-s}{16\rho}\right)+\Big[e^{\frac{1}{64\rho}}\Xi_\rho\left(\frac{1}{2}\right)+\mathrm{a}_\rho^-(s)\Big]\cosh\left(\frac{\frac{1}{2}-s}{16\rho}\right)\nonumber\\&e^{\frac{-s^2+{s}}{16\rho}}\tilde{\Xi}_{\rho}(s)=-\Big[e^{\frac{1}{64\rho}}\Xi_\rho\left(\frac{1}{2}\right)+\mathrm{a}_\rho^-(s)\Big]\sinh\left(\frac{\frac{1}{2}-s}{16\rho}\right)-\Big[\frac{e^{\frac{1}{32\rho}}}{4}\sqrt{\frac{\pi}{\rho}}-\mathrm{a}_\rho^+(s)\Big]\cosh\left(\frac{\frac{1}{2}-s}{16\rho}\right)\end{align*}
where 
$$\mathrm{a}^\pm_\rho(s)\hspace{-0.05cm}:=\hspace{-0.05cm}\frac{1}{32\rho}\int_{1/2}^{s}\d{t}\Big(e^{\frac{\frac{1}{2}-t}{16\rho}}\pm{e}^{\frac{t-\frac{1}{2}}{16\rho}}\Big)e^{\frac{-t^2+{t}}{16\rho}}\mathcal{M}\big[{e}^{-\rho\ln^2(\cdot)}\Delta_4\Psi\big]\hspace{-0.05cm}\left(\frac{t}{2}\right)$$
and we have $\mathrm{a}^\pm_\rho(s)\hspace{-0.05cm}=\hspace{-0.05cm}\mp\mathrm{a}^\pm_\rho(1-s)$.

Also \ref{SequendDiffSol} is an iterative formalism if we apply the same procedure with $\Psi$ replaced by $\Delta_\alpha\Psi$, more precise we have the following formulas:
\begin{theorem}\label{SequendDiffSolFin} Set $\mathcal{P}_\rho^{0}(s)=\cosh\left(\frac{\frac{1}{2}-s}{16\rho}\right)$, $\mathcal{P}_\rho^1(s):=\int_{1/2}^{s}\d{t}\sinh\left(\frac{s-t}{16\rho}\right)\cosh\left(\frac{\frac{1}{2}-t}{16\rho}\right)$ and $\mathcal{I}_\rho^1(s):=\int_{1/2}^{s}\d{t}\sinh\left(\frac{s-t}{16\rho}\right)\mathcal{M}\big[{e}^{-\rho\ln^2(\cdot)}\Delta^n_4\Psi\big]\left(\frac{t}{2}\right)$ and define $\mathcal{P}_\rho^n(s)$ and $\mathcal{I}_\rho^n(s)$ for $n>2$ by
$$\mathcal{P}_\rho^n(s):=\hspace{-0.05cm}\int_{1/2}^{s}\hspace{-0.3cm}\d{t_1}\sinh\left(\frac{s-t_1}{16\rho}\right)\hspace{-0.05cm}\int_{1/2}^{t_1}\hspace{-0.3cm}\d{t_2}\sinh\left(\frac{t_1-t_2}{16\rho}\right)\cdots\hspace{-0.05cm}\int_{1/2}^{t_{n-1}}\hspace{-0.55cm}\d{t_n}\sinh\left(\frac{t_{n-1}-t_n}{16\rho}\right)\cosh\left(\frac{\frac{1}{2}-t_n}{16\rho}\right)$$
$$\mathcal{I}_\rho^n(s):=\hspace{-0.05cm}\int_{1/2}^{s}\hspace{-0.3cm}\d{t_1}\sinh\left(\frac{s-t_1}{16\rho}\right)\cdots\hspace{-0.05cm}\int_{1/2}^{t_{n-1}}\hspace{-0.55cm}\d{t_n}\;\sinh\left(\frac{t_{n-1}-t_n}{16\rho}\right)e^{\frac{-t_{n}^2+{t_{n}}}{16\rho}}\mathcal{M}\big[{e}^{-\rho\ln^2(\cdot)}\Delta^n_4\Psi\big]\left(\frac{t_{n}}{2}\right)$$
We have $\mathcal{I}_\rho^n(1-s)=\mathcal{I}_\rho^n(s)$ and $\mathcal{P}_\rho^n(1-s)=\mathcal{P}_\rho^n(s)$ and for $n>1$ the formula
\begin{align*} e^{\frac{-s^2+{s}}{16\rho}}\mathcal{M}\big[\Psi{e}^{-\rho\ln^2(\cdot)}\big]\left(\frac{s}{2}\right)&=\frac{e^{\frac{1}{32\rho}}}{4}\sqrt{\frac{\pi}{\rho}}\sinh\left(\frac{\frac{1}{2}-s}{16\rho}\right)+\sum_{i=0}^{n-1}\frac{\mathcal{M}\big[{e}^{-\rho\ln^2(\cdot)}\Delta^{i}_4\Psi\big]\left(\frac{1}{4}\right)}{(16\rho)^i}\mathcal{P}_\rho^i(s)+\frac{\mathcal{I}_\rho^n(s)}{(16\rho)^n}\end{align*}
\end{theorem}
\begin{proof}[Proof] The equation $\Delta_4\Psi(t)=t^{-1/2}(\Delta_\alpha)\Psi(1/t)$ implies
\begin{align}\label{PBN}e^{\frac{-s^2+{s}}{16\rho}}\mathcal{M}\big[{e}^{-\rho\ln^2(\cdot)}\Delta^{n+1}_4\Psi\big]\left(\frac{s}{2}\right)=&e^{\frac{1}{64\rho}}\mathcal{M}\big[{e}^{-\rho\ln^2(\cdot)}\Delta^{n+1}_4\Psi\big]\left(\frac{1}{4}\right)\cosh\left(\frac{\frac{1}{2}-s}{16\rho}\right)\\&+\frac{1}{16\rho}\int_{1/2}^{s}\d{t}\;\sinh\left(\frac{s-t}{16\rho}\right)e^{\frac{-t^2+{t}}{16\rho}}\mathcal{M}\big[{e}^{-\rho\ln^2(\cdot)}\Delta^{n+2}_4\Psi\big]\left(\frac{t}{2}\right)\nonumber\end{align}\vspace{-0.4cm}\end{proof}
Set $\mathcal{P}_\rho^{n}=0$ if $n<0$ and notice the iterated integrals $\mathcal{P}_\rho^{n}$ satisfy $\big(\id-(16\rho)^2\frac{\partial^2}{\partial{s}^2}\big)\mathcal{P}_\rho^{n}(s)=16\rho\mathcal{P}_\rho^{n-1}(s)$, hence $\big(\id-(16\rho)^2\frac{\partial^2}{\partial{s}^2}\big)^{n+1}\mathcal{P}_\rho^{n}(s)=0$ and this implies that $\mathcal{P}_\rho^{n}(s)=p_\rho^{A}(s)\sinh\big(\frac{\frac{1}{2}-s}{16\rho}\big)+p_\rho^{B}(s)\cosh\big(\frac{\frac{1}{2}-s}{16\rho}\big)$ where $p_\rho^{A}(s)=-p_\rho^{A}(1-s)$ and $p_\rho^{B}(s)=p_\rho^{B}(1-s)$ are polynomials of degree $<n+1$. $\mathcal{P}_\rho^{n}(s)$ can be computed in principle, we have for example the calculation
$\mathcal{P}_\rho^{1}(s)=\cosh\left(\frac{\frac{1}{2}-s}{16\rho}\right)\int_{1/2}^{s}\d{t_1}\frac{\sinh\left(\frac{1-2t_1}{16\rho}\right)}{2}-\sinh\left(\frac{\frac{1}{2}-s}{16\rho}\right)\int_{1/2}^{s}\d{t_1}\frac{1+\cosh\left(\frac{1-2t_1}{16\rho}\right)}{2}=\frac{(\frac{1}{2}-s)\sinh\big(\frac{\frac{1}{2}-s}{16\rho}\big)-8\rho\cosh\big(\frac{\frac{1}{2}-s}{16\rho}\big)}{2}$.

\vfill{\begin{center}
\textcircled{c} Copyright by Johannes L\"offler, 2015, All Rights Reserved
\end{center}}

\end{document}